\documentclass{amsart}%
\usepackage{amssymb}
\usepackage{amsmath}
\usepackage{latexsym}
\usepackage{hyperref}
\usepackage{amsfonts}
\usepackage{graphicx}%
\newtheorem{theorem}{Theorem}

\newcommand{\be}{\begin{equation}}
\newcommand{\ee}{\end{equation}}
\newcommand{\bea}{\begin{eqnarray}}
\newcommand{\eea}{\end{eqnarray}}

\begin{document}
\title{A lower bound for the scalar curvature of certain steady gradient Ricci solitons}
\author{Bennett Chow}
\author{Peng Lu}
\author{Bo Yang$^{\text{1}}$}
\maketitle


We have the following result regarding steady Ricci solitons. See
\cite{Brendle}, \cite{CaoChen}, \cite{GuoHongxin}, \cite{H3}, and
\cite{MontanoSesum} for some earlier works on the qualitative aspects of
steady Ricci solitons.\footnotetext[1]{Addresses. Bennett Chow and Bo Yang:
Math. Dept., UC San Diego; Peng Lu: Math. Dept., U of Oregon.}

\begin{theorem}
Let $\left(  \mathcal{M}^{n},g,f\right)  $ be a complete steady gradient Ricci
soliton with $R_{ij}=-\nabla_{i}\nabla_{j}f$ and $R+\left\vert \nabla
f\right\vert ^{2}=1$. If $\lim_{x\rightarrow\infty}f\left(  x\right)
=-\infty$ and $f\leq0$, then $R\geq\frac{1}{\sqrt{\frac{n}{2}}+2}e^{f}$.
\end{theorem}

\begin{proof}
Define $\Delta_{f}=\Delta-\nabla f\cdot\nabla$. Then $\Delta_{f}f=-1$,
$\Delta_{f}R=-2\left\vert \operatorname{Rc}\right\vert ^{2}\leq-\frac{2}%
{n}R^{2}$, and $\Delta_{f}(e^{f})=-R\,e^{f}$. For $c\in\mathbb{R}$,%
\[
\Delta_{f}\left(  R-ce^{f}\right)  \leq-\frac{2}{n}R^{2}+cR\,e^{f}\leq
\frac{nc^{2}}{8}e^{2f}.
\]
Using $\Delta_{f}(e^{2f})=2e^{2f}\left(  1-2R\right)  $, we compute for
$b\in\mathbb{R}$ that%
\begin{equation}
\Delta_{f}\left(  R-ce^{f}-be^{2f}\right)  \leq\left(  \frac{nc^{2}}%
{8}-2b+4bR\right)  e^{2f}.\label{Rumpelstiltskin}%
\end{equation}
Suppose $R-ce^{f}-be^{2f}$ is negative somewhere. Then, since $R\geq0$ by
\cite{ChenB} and $\lim_{x\rightarrow\infty}e^{f\left(  x\right)  }=0$ by
hypothesis, a negative minimum of $R-ce^{f}-be^{2f}$ is attained at some
point. By (\ref{Rumpelstiltskin}) and the maximum principle, at such a point
we have%
\[
0\leq\frac{nc^{2}}{8}-2b+4bR<\frac{nc^{2}}{8}-2b+4b\left(  c+b\right)
\]
since $f\leq0$. Given $c\in(0,\frac{1}{2}]$, the minimizing choice
$b=\frac{1-2c}{4}$ yields $\frac{\left(  1-2c\right)  ^{2}}{4}<\frac{nc^{2}%
}{8}$. We obtain a contradiction by choosing $c=\frac{1}{\sqrt{\frac{n}{2}}%
+2}$.
\end{proof}

\noindent\textbf{Remark.} \emph{Given }$O\in\mathcal{M}$, \emph{since}
$\left\vert \nabla f\right\vert \leq1$, \emph{we have }$f\left(  x\right)
\geq f(O)-d\left(  x,O\right)  $ \emph{on} $\mathcal{\dot{M}}$. \emph{For the
cigar soliton} $(\mathbb{R}^{2},\frac{4(dx^{2}+dy^{2})}{1+x^{2}+y^{2}})$
\emph{we have} $R=e^{f}$ \emph{assuming} $\max_{x\in\mathbb{R}^{2}}f\left(
x\right)  =0$. \emph{See} \cite{PengWu} \emph{for an estimate for the
potential functions of steady gradient Ricci solitons.}

\markright{A CONDITION FOR RICCI SHRINKERS TO HAVE POSITIVE AVR}


\begin{thebibliography}{9}                                                                                                %


\bibitem {Brendle}Brendle, S. \emph{Uniqueness of gradient Ricci solitons.} arXiv:1010.3684.

\bibitem {CaoChen}Cao; H.-D.; Chen, Q. \emph{On locally conformally flat
gradient steady Ricci solitons.} Trans. AMS, to appear.

\bibitem {ChenB}Chen, B.-L. \emph{Strong uniqueness of the Ricci flow.} J. of
Diff. Geom. \textbf{82} (2009), 363--382.

\bibitem {GuoHongxin}Guo, H. \emph{Area growth rate of the level surface of
the potential function on the }$\emph{3}$\emph{-dimensional steady Ricci
soliton.} Proc. AMS \textbf{137} (2009), 2093--2097.

\bibitem {H3}Hamilton, R.\thinspace S. \emph{The formation of singularities in
the Ricci flow.\ }Surv. Diff. Geom. Vol.\ II, 7--136, Intern. Press,
Cambridge, MA, 1995.

\bibitem {MontanoSesum}Munteanu, O.; Sesum, N. \emph{On gradient Ricci
solitons.} arXiv:0910.1105.

\bibitem {PengWu}Wu, P. \emph{Remarks on gradient steady Ricci solitons.} arXiv:1102.3018.
\end{thebibliography}
\end{document}